\documentclass[12pt,reqno]{amsart}

\usepackage{epsfig}
\usepackage{amsmath, amsthm, amsfonts, amssymb, mathrsfs}
\usepackage{tikz-cd}
\usepackage{graphicx}

\numberwithin{equation}{section}
\DeclareMathOperator\Hom{Hom}

\newcommand{\D}{{\mathbb D}}

\newcommand{\R}{{\mathbb R}}
\newcommand{\C}{{\mathbb C}}
\newcommand{\N}{{\mathbb N}}

\newcommand{\Z}{{\mathbb Z}}
\newcommand{\Zsstar}{\mathbb{Z}_s^\times}

\newcommand{\F}{{\mathcal F}}

\newcommand{\rhot}{\rho_\theta}

\newtheorem{theo}{{\sc \bf Theorem}}[section]

\newtheorem{prop}[theo]{{\sc \bf Proposition}}

\begin{document}

\title{Spectral Triples for Hensel-Steinitz Algebras}

\author[Hebert]{Shelley Hebert}
\address{Department of Mathematics and Statistics,
Mississippi State University,
175 President's Cir. Mississippi State, MS 39762, U.S.A.}
\email{shebert@eastms.edu}

\author[Klimek]{Slawomir Klimek}
\address{Department of Mathematical Sciences,
Indiana University Indianapolis,
402 N. Blackford St., Indianapolis, IN 46202, U.S.A.}
\email{klimek@iu.edu}

\author[McBride]{Matt McBride}
\address{Department of Mathematics and Statistics,
Mississippi State University,
175 President's Cir., Mississippi State, MS 39762, U.S.A.}
\email{mmcbride@math.msstate.edu}

\date{\today}

\begin{abstract}
We explicitly construct Fredholm modules and spectral triples representing any element of $K$-homology groups of Hensel-Steinitz algebras.
\end{abstract}

\maketitle

\section{Introduction}

Hensel-Steinitz algebras, introduced in \cite{HKMP1}, are natural C$^*$-algebras associated to multiplication maps of the ring of $s$-adic integers $\Z_s$. A Hensel-Steinitz algebra is defined as a crossed product of $C(\Z_s)$ by a monomorphism with hereditary range defined in equation \eqref{alpha_def} below. It should be noted that Hensel-Steinitz algebras, despite similar definitions, are quite different than the algebras introduced in \cite{CuntzVershik}. Hensel-Steinitz algebras are one of many interesting examples of C$^*$-algebras with connections to number theory, see for example \cite{CuLi}, \cite{HKMP2}.

Hensel-Steinitz algebras can be related to some familiar C$^*$-algebras which allows for computation of their $K$-theory and $K$-homology groups. However, it is often not easy to pinpoint projections and unitaries generating the $K$-theory and Fredholm modules or spectral triples generating the $K$-homology groups. The purpose of this note is to describe in detail such Fredholm modules and spectral triples representing any element of $K$-homology groups of Hensel-Steinitz algebras. The $K$-theory was already sufficiently described in \cite{HKMP1}. While  $K$-theory and $K$-homology groups of C$^*$-algebras are part of their non-commutative topology, constructions of explicit representatives belong to their non-commutative geometry.

The first part of this note is a short summary of some of the results from \cite{HKMP1}.
The second part on $K$-homology, Fredholm modules, and spectral triples for Hensel-Steinitz algebras is new, however it contains ideas and techniques inspired by a similar study in \cite{RDOdo}.

\section{Hensel-Steinitz Algebras}

For $s\in\N$ with $s\ge2$, the space of $s$-adic integers $\Z_s$ is defined as consisting of infinite sums:
\begin{equation*}
\Z_s= \left\{x=\sum_{j=0}^\infty x_j s^j: 0\le x_j\le s-1, j=0,1,2,\ldots\right\}\,.
\end{equation*}
We call the above expansion of $x$ its $s$-adic expansion.  Note that $\Z_s$ is isomorphic with the countable product $\prod \Z/s\Z$ and so, when the product is equipped with the Tychonoff topology, it is a Cantor set. It is also a metric space with the usual norm, $|\cdot|_s$, which is defined as follows: 
\begin{equation*}
|0|_s=0 \textrm{ and, if }x\ne 0, |x|_s=s^{-n},
\end{equation*}
where $x_n$ is the first nonzero term in the above $s$-adic expansion.
Moreover, $\Z_s$ is an Abelian ring with unity with respect to addition and multiplication with a carry,
for more details, see \cite{monothetic}.

The main object of interest in this note is the map $\alpha:C(\Z_s)\to C(\Z_s)$ defined by
\begin{equation}\label{alpha_def}
(\alpha f)(x) = \left\{
\begin{aligned}
&f\left(\frac{x}{s}\right) &&\textrm{ if }s|x \\
&0 &&\textrm{else.}
\end{aligned}\right.
\end{equation}
Notice that
\begin{equation*}
\alpha (1)(x) = \left\{
\begin{aligned}
&1&&\textrm{ if }s|x \\
&0 &&\textrm{else}
\end{aligned}\right.
\end{equation*}
is the characteristic function of the subset of the ring $\Z_s$ consisting of those elements that are divisible by $s$. It follows from equation \eqref{alpha_def} that the range of $\alpha$ is equal to $\alpha(1)C(\Z_s)$. Thus $\alpha$ is a monomorphism with hereditary range and in particular, all well known definitions of a crossed product by an endomorphism coincide in this case, see \cite{Mu}, \cite{St}, \cite{E}. 
With this, the Hensel-Steinitz algebra, denoted by $HS(s)$ and  defined in  \cite{HKMP1}, is the  following crossed product:
\begin{equation*}
HS(s) := C(\Z_s)\rtimes_\alpha\N.
\end{equation*}

It is useful to realize $HS(s)$ as a concrete C$^*$-algebra. Let $H=\ell^2(\Z)$ and $\{E_l\}_{l\in\Z}$ be its canonical basis, then since $\Z$ is a dense subset of $\Z_s$, the mapping $C(\Z_s)\to B(H)$, $f\mapsto M_f$, given by 
$$M_fE_l = f(l)E_l$$
 is a faithful representation of $C(\Z_s)$.  

Let $V$ be the $s$-adic shift operator defined on $H$ via
\begin{equation*}
VE_l = E_{sl}\,.
\end{equation*}
A simple calculation verifies that
\begin{equation*}
V^*E_l = \left\{
\begin{aligned}
&E_{l/s} &&\textrm{ if } s|l \\
&0 &&\textrm{else.}
\end{aligned}\right.
\end{equation*}

There is an isomorphism between C$^*$-algebras, see  \cite{HKMP1} for details:
\begin{equation*}
HS(s) = C(\Z_s)\rtimes_\alpha\N \cong C^*(V,M_f:f\in C(\Z_s))\,.
\end{equation*}
This follows because $V$ and $M_f$ satisfy the relations of the Stacey's crossed product in \cite{St}: 
\begin{equation}\label{StaceyRel}
V^*V=I\textrm{ and }VM_fV^*=M_{\alpha f}\,
\end{equation}
Additionally, the algebra $C^*(V,M_f:f\in C(\Z_s))$ carries a natural gauge action defined by maps $\rhot$  with $\theta \in \R/\Z$ which is given by
\begin{equation}\label{rhotheta}
\rhot (a)=e^{2\pi i\theta \mathbb P} a e^{-2\pi i\theta\mathbb P}.
\end{equation}
Here, given $l\in\Z$ and writing $l=s^ml'$ with $l'$ not divisible by $s$, the diagonal operator $\mathbb P$ is defined by:
\begin{equation*}
\mathbb{P}E_{s^ml'} = mE_{s^ml'},\ \ \ \mathbb{P}E_{0} = 0.
\end{equation*}
The gauge action satisfies $\rhot(V)=e^{2\pi i\theta}V$ and that $\rhot(M_f)=M_f$.  
Then,  the O'Donovan condition for crossed products implies that we have $HS(s) \cong C^*(V,M_f)$.

\section{The Structure of Hensel-Steinitz Algebras}

To understand the structure of $HS(s)$ we first look at an ideal, denoted by $I_s$, which is the kernel of the representation $\pi_0$ in $B(H)$ defined by $\pi_0(M_f)=f(0)I$ and $\pi_0(V)=v$, where $v \in B(H)$ is the standard bilateral shift 
$$vE_l = E_{l+1},$$ 
so that
$$
I_s : = \textup{Ker } \pi_0.
$$

The ideal $I_s \subseteq HS(s)$ can also be described as the subalgebra generated by the operators $V^m M_{f}$, $m\geq 0$
and $M_{f}(V^*)^{-m}$, $m<0$ such that $f\in C(\Z_s)$ with $f(0)=0$.

Let $\Z_s^\times$ be the unit sphere in $\Z_s$, and notice that if $x\in \Z_s$ has the $s$-adic expansion:
\begin{equation*}
x=x_0 + x_1s + x_2s^2+\cdots
\end{equation*}
then, $x\in\Z_s^\times$ if and only if $x_0\neq0$. Observe that $\Z_s^\times$ is both a closed and open subset of $\Z_s$. Let $\mathcal{K}$ be the C$^*$-algebra of compact operators in a separable Hilbert space,   the following isomorphism of C$^*$-algebras is shown in \cite{HKMP1}:
\begin{equation*}
I_s\cong C(\Z_s^\times)\otimes\mathcal{K}\,.
\end{equation*}

We outline an argument from that paper since some of the notation is useful later. Given nonzero integer $l$, decompose $l=s^ml'$, with $l'\in\Z\subseteq\Z_s$, $|l'|_s=1$. For $\lambda \in C(\Zsstar)$, define the multiplication operator $m_\lambda :H\to H$ by the following
\begin{equation}\label{Mphi}
m_\lambda E_l=\left\{
\begin{aligned}
&\lambda(l')E_l &&\textrm{if }l\ne 0 \\
&0 &&\textrm{if }l=0.
\end{aligned}\right.
\end{equation}
Similarly, given a sequence $\chi=\{\chi(l)\} \in c_0(\Z_{\ge0})$ and the decomposition above, define $\mu_\chi:H\to H$ by
\begin{equation}\label{Muchi}
\mu_\chi E_l = \left\{
\begin{aligned}
&\chi(m)E_l &&\textrm{if }l\ne 0 \\
&0 &&\textrm{if }l=0.
\end{aligned}\right.
\end{equation}
Then, we have the following isomorphisms of $C^*$-algebras:
\begin{equation*}
C^*(M_f : f\in C(\Z_s), f(0)=0) \cong C^*(m_\lambda \mu_\chi : \lambda\in C(\Zsstar), \chi\in c_0(\Z_{\ge0}))\cong C(\Zsstar)\otimes c_0(\Z_{\ge0})\,.
\end{equation*}

Since $m_\lambda$ commutes with $V$ and $\mu_\chi$, the following two C$^*$-algebras are isomorphic:
\begin{equation*}
\begin{aligned}
&C^*(V^nM_{f}:f\in C(\Z_s), f(0)=0, n\in \Z_{\ge0}) \\  &C^*(m_{\lambda}:\lambda\in C(\Zsstar))\otimes C^*(V^n\mu_\chi: n\in \Z_{\ge0}, \chi\in c_0(\Z_{\ge0})).
\end{aligned}
\end{equation*}

Finally, we have the following isomorphism:
\begin{equation*}
C^*(V^n\mu_\chi: n \in \Z_{\ge0}, \chi\in c_0(\Z_{\ge0})) \cong \mathcal{K}\,,
\end{equation*}
since the matrix units can be easily constructed from $V^n\mu_\chi$ and they generate the algebra.

This leads to the short exact sequence describing the structure of $HS(s)$:
\begin{equation}\label{HSStructure}
0 \to C(\Z_s^\times)\otimes\mathcal{K} \to HS(s) \to C(\R/\Z) \to 0,
\end{equation}
which is instrumental in further analysis of the Hensel-Steinitz algebra. 

\section{Smooth Subalgebras}

Our next goal is to introduce natural smooth subalgebras of Hensel-Steinitz algebras $HS(s)$, see  \cite{HKMP1}.
We first define a smooth subalgebra of the ideal $I_s$. This is done by using Fourier coefficients of elements of $HS(s)$ with respect to the gauge action given by equation \eqref{rhotheta}.

We again use the decomposition of nonzero elements of $\Z_s$ as $s^m x$, $m\in \Z_{\ge0}$, $|x|_s=1$ and, given $f\in C(\Z_s)$ with $f(0)=0$, we put 
$$F(m,x)=f(s^m x).$$   

The smooth subalgebra $I_s^\infty$ of $I_s$ is defined as the space of $a\in I_s$ with a norm convergent series expansion:
\begin{equation}\label{IsSeries}
a=\sum_{n\ge0}V^nM_{F_n} + \sum_{n<0}M_{F_n}(V^*)^{-n},   
\end{equation}
where $F_n\in C_0(\Z_{\geq 0}\times\Zsstar)$ satisfy a rapid decay (RD) condition:
\begin{equation*}
I_s^\infty := \left\{a\in I_s: \{\underset{x}{\textrm{sup}}|F_n(m,x)|\}_{m,n}\textrm{ is RD}\right\}\,.
\end{equation*}

Recall from Proposition 2.1 in \cite{KMP3}, that the space of smooth compact operators $\mathcal{K}^{\infty}$  can be described as
\begin{equation*}
\mathcal{K}^{\infty} = \left\{ \sum_{n\ge0} V^n x_n + \sum_{n< 0} x_n (V^*)^{-n} : x_n\in c_0(\Z_{\geq 0}), \{x_n(m)\}_{m,n} \text{ is RD} \right\}\,.
\end{equation*} 
We have the following isomorphism of algebras:
\begin{equation*}
I_s^{\infty} \cong C(\Zsstar) \otimes \mathcal{K}^{\infty}.
\end{equation*}

Notice that the elements of $I_s^{\infty}$ don't have any extra regularity in the $\Zsstar$ coordinates, which is required for the construction of spectral triples below. For that purpose, the simplest requirement is  Lipschitz continuity.

Let $f:\Zsstar\rightarrow \C$ be a Lipschitz function and define 
$$L(f)=\textrm{inf}\{k: |f(x)-f(y)|\leq k |x-y|_s)\},$$ 
called the Lipschitz constant of $f$. Let $C_L(\Zsstar)$ be the Banach $*$-algebra of Lipschitz functions on $\Zsstar$ equipped with the norm:
\begin{equation*}
\|f\|_L:=\underset{x}{\textrm{sup}}|f(x)|+L(f).
\end{equation*}

Using the notation from equation \eqref{IsSeries}, we define
\begin{equation*}
I_{s,L}^\infty := \left\{a\in I_s: \{\|F_n(m,\cdot)\|_L\}_{m,n}\textrm{ is RD}\right\}\,.
\end{equation*}
so that
\begin{equation*}
I_{s,L}^{\infty} \cong C_L(\Zsstar) \otimes \mathcal{K}^{\infty}.
\end{equation*}

For $\phi\in C^\infty(\R/\Z)$, a Toeplitz operator $T(\phi)$ acting on $H$ is given by
\begin{equation}\label{Tseries}
T(\phi) = \sum_{m\geq 0}\phi_mV^m + \sum_{m<0}\phi_m(V^*)^{-m}
\end{equation}
where $\phi_m$ are the $m^{\textrm{th}}$ Fourier coefficients of $\phi$.  Since $\phi\in C^\infty(\R/\Z)$, it follows that the sequence $\{\phi_m\}$ is RD and the above series is norm convergent.

We define the smooth Hensel-Steinitz algebra, $HS^\infty(s)$ by
\begin{equation*}
HS^\infty(s) = \{a\in HS(s): a=T(\phi)+c \text{ with } \phi\in C^\infty(\R/\Z), c\in I_s^\infty\}\,.
\end{equation*}
It was verified in \cite{HKMP1} that $HS^\infty(s)$ is a smooth subalgebra of $HS(s)$ and in particular it is stable under the smooth functional calculus of self-adjoint elements. Below we also need the following $*$-subalgebra of $HS(s)$:
\begin{equation*}
HS^\infty_L(s) = \{a\in HS(s): a=T(\phi)+c \text{ with } \phi\in C^\infty(\R/\Z), c\in I_{s,L}^\infty\}\,.
\end{equation*}
Notice that this is a smooth, dense subalgebra of $HS(s)$ and we will refer to it as a smooth-Lipschitz subalgebra.

\section{K-Theory}

We review the $K$-theory of $HS(s)$ from \cite{HKMP1}. It can be computed from the six-term exact sequence in $K$-theory induced by the short exact sequence given by equation \eqref{HSStructure}.

Denote by $C(\Zsstar,\Z)$ the group of continuous functions on $\Zsstar$ with values in $\Z$.
Since $I_s \cong C(\Zsstar) \otimes \mathcal{K}$, the stability of $K$-theory and the known $K$-theory of continuous functions on totally disconnected compact spaces shows that 
\begin{equation*}
K_0(I_s) \cong K_0(C(\Zsstar))\cong C(\Zsstar,\Z) \quad \text{and} \quad K_1(I_s) \cong K_1(C(\Zsstar))\cong 0.
\end{equation*}
The explicit stability isomorphism of $K_0(C(\Zsstar)) \cong K_0(I_s)$ is given by 
\begin{equation*}
p\mapsto p\otimes (I-VV^*).
\end{equation*} 

The six-term exact sequence in $K$-theory simplifies to:
\begin{equation*}
0 \to \Z \to C(\Z_s^\times, \Z) \to K_0(HS(s)) \to \Z \to 0.
\end{equation*}
By \cite{HKMP1}, the range of the map $\Z \to C(\Zsstar , \Z)$ is precisely the constant functions in $C(\Zsstar, \Z)$. 
Taking the quotient, we obtain the short exact sequence
\begin{equation*}
0 \to C(\Zsstar,\Z)/\Z \to K_0(HS(s)) \to \Z \to 0,
\end{equation*}
which can be easily seen to split. This implies that the $K$-theory of $HS(s)$ is given by 
\begin{equation*}
K_0(HS(s)) \cong (C(\Zsstar,\Z)/\Z) \oplus \Z  \quad \text{and} \quad K_1(HS(s)) \cong 0. 
\end{equation*}
It is convenient to naturally identify $C(\Zsstar,\Z)/\Z$ with $C_{\{1\}}(\Zsstar,\Z)$ where $C_{\{1\}}(\Zsstar,\Z)$ are those functions in $C(\Zsstar,\Z)$ which vanish at $1\in \Zsstar$:
\begin{equation*}
C_{\{1\}}(\Zsstar,\Z) = \{f\in C(\Zsstar,\Z) : f(1)=0\}.
\end{equation*}

To proceed, we need the generators of $K_0(HS(s))$. 
From the above calculation we see that, other than the class of identity, all other $K_0(HS(s))$ classes are generated by the projections 
\begin{equation}\label{Kgenerators}
m_{1_X}(I-VV^*)\in I_s,
\end{equation}
where $X$ is a closed and open subset of $\Zsstar\subseteq \Z_s$ not containing $1\in \Zsstar$, and $1_X$ is its characteristic function.

\section{K-Homology}

From reference \cite{BB}, Section 22.3, we know that the bootstrap category $\mathcal{N}$ of separable nuclear C$^*$-algebras contains all commutative C$^*$-algebras, all AF C$^*$-algebras, and is closed under tensor products and extensions. It follows that $HS(s)$ is in $\mathcal{N}$ and we can use the Universal Coefficient Theorem of \cite{RSUCT}. Since $C_{\{1\}}(\Zsstar, \Z)$ is free (see below and also exercise 7.7.5 in \cite{HR}) the $K$-homology can be identified with the groups of of homomorphisms from the $K$-theory to $\Z$ and we obtain 
$$K^1(HS(s)) \cong 0 \quad \text{and} \quad K^0(HS(s)) \cong \Hom(C_{\{1\}}(\Zsstar, \Z)\oplus\Z,\Z)\cong \Hom(C_{\{1\}}(\Zsstar, \Z),\Z)\oplus\Z.$$
In the formula above, the last factor of $\Z$ is generated by the homomorphism $\eta: K_0(HS(s))\to\Z$ such that
\begin{equation}\label{etadef}
\eta(I)=1\ \ \textrm{ and }\ \ \eta(I-VV^*)=0.
\end{equation}
It remains to study $\Hom(C_{\{1\}}(\Zsstar, \Z),\Z)$.

\subsection{Generators}

Next, we describe a convenient set of generators of $C_{\{1\}}(\Zsstar, \Z)$. For $n=0,1,2, \ldots$ and $0\leq x<s^n$ consider the following functions on $\Z_s$:
\begin{equation*}
1_{(n,x)}(z) = \left\{
\begin{aligned}
&1 &&\textrm{if }s^n|(z-x)  \\
&0 &&\textrm{else.}
\end{aligned}\right.
\end{equation*}
Those are characteristic functions of open and closed subsets of $\Z_s$. Any two of those sets are either disjoint or one is a subset of the other.  Moreover, linear combinations of functions $1_{(n,x)}$ are dense in $C(\Z_s)$ and any function in $C(\Z_s, \Z)$ is a finite integer combination of $1_{(n,x)}$'s. However, $1_{(n,x)}$'s are not independent since
\begin{equation*}
1_{(n,x)} = \sum_{l=0}^{s-1} 1_{(n+1,x+ls^n)}.
\end{equation*}

Since $\Zsstar$ is both closed and open in $\Z_s$ we can think of $C(\Zsstar)$ as the subspace of $C(\Z_s)$ of functions vanishing outside of $\Zsstar$. From the definition, if $x\in \Z_s$ has the $s$-adic expansion:
\begin{equation*}
x=x_0 + x_1s + x_2s^2+\cdots
\end{equation*}
then, $x\in\Z_s^\times$ if and only if $x_0\neq0$. It follows that $1_{(n,x)}\in C(\Zsstar)$ if and only if $s$ does not divide $x$ and $1_{(n,x)}\in C_{\{1\}}(\Zsstar)$ if additionally $x>1$.

To proceed further we consider the set $T\subseteq\Z$ given by
\begin{equation*}
T:=\{x\in\Z: x\geq 2, s\nmid x\}.
\end{equation*}
For $x\in T$ we define
\begin{equation*}
1_{(x)}:= 1_{(n(x),x)}, 
\end{equation*}
where $n(x)\in\N$ is defined so that $s^{n(x)-1}< x<s^{n(x)}$.

\begin{prop} The collection of functions $\{1_{(x)}\}$ for all $x\in T$ form free generators of the group $C_{\{1\}}(\Zsstar, \Z)$. In other words,
any element $f\in C_{\{1\}}(\Zsstar, \Z)$ can be uniquely written as a finite sum:
\begin{equation*}
f=\sum_{x\in T} f_{(x)}\,1_{(x)}
\end{equation*}
where $f_{(x)}\in\Z$.
\end{prop}
\begin{proof} 
As noticed above, linear combinations of the $1_{(n,x)}$'s with $x\in T$ generate $C_{\{1\}}(\Zsstar)$ and since continuity of a function in $C_{\{1\}}(\Zsstar, \Z)$ is equivalent to the function being locally constant, any function in $C_{\{1\}}(\Zsstar, \Z)$ is a linear (integer) combination of $1_{(n,x)}$'s with $x\in T$.  Thus, to establish the existence part of the proposition we need to express the functions $1_{(n,x)}$ in terms of the functions $1_{(x)}$'s.  

We proceed by induction in $n$.  The base case of $n=1$ is immediate by definition.
For the inductive step, suppose $1_{(n,x)}$ can be expressed in terms of the $1_{(x)}$'s and consider the function $1_{(n+1,x)}$.  If $s^n< x<s^{n+1}$, then by definition we have that $1_{(n+1,x)}=1_{(x)}$.  On the other hand if $2\le x<s^n$,  we have that
\begin{equation*}
1_{(n+1,x)} = 1_{(n,x)} - \sum_{l=1}^{s -1}1_{(n+1,x+ls^n)} = 1_{(n,x)} -\sum_{l=1}^{s-1}1_{(x+ls^n)},
\end{equation*}
which, by assumption, shows the inductive step and hence existence.

To establish uniqueness suppose that
$$\sum_{x\in T} f_{(x)}\,1_{(x)}(z)=0.$$ 
We evaluate the sum on the left-hand side of the above equation at $z=2,3,\ldots$ and prove recursively that $f_{(x)}=0$.

Consider
\begin{equation*}
0=\sum_{x\in T} f_{(x)}1_{(x)}(2) = \sum_{x\in T} f_{(x)}1_{(n(x),x)}(2)
\end{equation*}
for $s^{n(x)-1}\le x < s^{n(x)}$.  Note that $1_{(n(x),x)}(2) =1$ if $s^{n(x)}$ divides $x-2$ and is zero else.  Since $x-2<s^{n(x)}$, it follows that from the above sum that $f_{(2)}=0$.  Following this procedure recursively, we conclude that $f_{(x)}=0$ for all $x\in T$ and hence the uniqueness follows.

\end{proof}

\subsection{Special Homomorphisms}

There are two classes of particularly interesting elements of the group $\Hom(C_{\{1\}}(\Zsstar, \Z),\Z)$.
Define $e_{(y)}\in \Hom(C_{\{1\}}(\Zsstar, \Z),\Z)$ for $y\in T$ to be the ``dual" homomorphisms to the generators $1_{(x)}$ of $C_{\{1\}}(\Zsstar, \Z)$:
\begin{equation*}
e_{(y)}(1_{(x)}) = \left\{
\begin{aligned}
&1 &&\textrm{if }y=x  \\
&0 &&\textrm{else.}
\end{aligned}\right.
\end{equation*}
The significance of $e_{(y)}$ is that any homomorphism $\Phi\in \Hom(C_{\{1\}}(\Zsstar, \Z),\Z)$ can be uniquely written as a possibly infinite sum:
\begin{equation*}
\Phi=\sum_{y\in T}\phi_{(y)}\,e_{(y)}
\end{equation*}
where $\phi_{(y)}\in\Z$. The infinite sum above is always well-defined because when evaluating $\Phi$ on $f\in C_{\{1\}}(\Zsstar, \Z)$ the resulting sum is always finite.

Another important class of homomorphisms are the evaluation maps $\delta_z$, $z=0,1,2,\ldots$:
\begin{equation*}
\delta_z(f):=f(z).
\end{equation*}
Notice that if $z\notin T$ but $x\in T$ then $\delta_z(1_{(x)}) = 0$ and so, for $z\notin T$, $\delta_z$ is a zero homomorphism on $C_{\{1\}}(\Zsstar, \Z)$. However, below it is still convenient to use that notation of evaluations at $z\not\in T$.

Let $x$ be a natural number, if $n(x)$ is such that $s^{n(x)-1}\leq x<s^{n(x)}$ we set:
\begin{equation*}
\gamma(x):= x \textrm{ mod } s^{n(x)-1},
\end{equation*}
so that $0\leq \gamma(x)<  s^{n(x)-1}$.
Notice that there are $x\in T$ such that $\gamma(x)\not\in T$. In fact, from the definition, $\gamma(x)=0$ for $x=2,3,\ldots s-1$. Additionally, $\gamma(x)=1$ if $x$ is of the form 
$$x=1+s^{n-1}+ys$$ 
for $n\geq 2$ and $y=0,1,\ldots, s^{n-1}-s^{n-2}-1$. In all other cases if $x\in T$ then $\gamma(x)\in T$.

The two classes of homomorphisms are closely related by the following formulas.
\begin{prop}\label{edeltaprop}
With the above notation we have
\begin{equation*}
\delta_z=e_{(z)}+e_{(\gamma(z))}+e_{(\gamma^2(z))}+\cdots +e_{(\gamma^i(z)) }+\cdots \textrm{ and } e_{(z)}=\delta_z-\delta_{\gamma(z)}.
\end{equation*}
In the first formula, the sum is such that $\gamma^i(z)\in T$. In the second formula, the homomorphism $\delta_{\gamma(z)}$ might be zero.
\end{prop}
\begin{proof} 
Fix $z\in T$. For $s^{n(x)-1}< x <s^{n(x)}$ we have that
\begin{equation*}
\delta_z(1_{(x)}) = \delta_z(1_{(n(x),x)}) = \left\{
\begin{aligned}
&1 &&\textrm{ for }s^{n(x)}|(z-x) \\
&0 &&\textrm{ else}.
\end{aligned}\right.
\end{equation*}
We want to determine for which $x$ we have $\delta_z(1_{(x)})$ is non-zero. First observe that to get nontrivial $\delta_z(1_{(x)})$, we must have $x\le z$ since otherwise $0<x-z<s^{n(z)}$ and so $s^{n(z)}$ cannot divide $x-z$.  In particular, for a fixed $z$, there are only finitely many $x$ such that $\delta_z(1_{(x)})\ne 0$. We determine them inductively in decreasing order. 

The first is $x=z$ and it is the only one such that $s^{n(x)}=s^{n(z)}$.  The next possible $s^{n(x)}$ is $s^{n(z)-1}$ and for such a choice we must have $s^{n(z)-1}$ dividing $x-z$ and hence 
$$x=z\textrm{ mod }s^{n(z)-1} = \gamma(z).$$ 
Repeating this reductive process results in the first formula.  The second formula is a direct consequence of the first.
\end{proof}

\section{Fredholm Modules}

The goal of this section is to describe in detail Fredholm modules representing any element of the $K$-homology groups of Hensel-Steinitz algebras.

\subsection{Definitions}

An odd Fredholm module over a unital C$^*$-algebra $A$, see for example \cite{HR}, is a triple $(H, \rho , \F)$ where $H$ is a Hilbert space, $\rho: A \to B(H)$ is a $*$-representation, and $\F \in B(H)$ satisfies 
$$\F^*- \F,\ I - \F^2 \in \mathcal{K}(H),\textrm{ and }[\F, \rho(a)] \in \mathcal{K}(H)$$ 
for all $a \in A$, where $ \mathcal{K}(H)$ is the algebra of compact operators in $B(H)$. Similarly, an even Fredholm module is the above information, together with a $\Z_2$-grading $\Gamma$ of $H$, such that 
$$\Gamma^* = \Gamma,\ \Gamma^2 = I,\ \Gamma \F = - \F \Gamma,\ \textrm{and  }\Gamma \rho(a) = \rho(a) \Gamma$$ 
for all $a \in A$.  It follows that $H=H_{ev} \oplus H_{odd}$, $\rho_{ev} \oplus \rho_{odd}$ . Without the loss of generality, we can assume that $\F$ is self-adjoint, so it has the form:
\begin{equation}\label{FGmatrix}
 \F = \begin{bmatrix} 0 & G \\G^* & 0 \end{bmatrix}
\end{equation}
where $G$ maps $H_{odd}$ to $H_{ev}$.
Classes of odd Fredholm modules over $A$ form the K-homology group $K^1(A) : = KK^1(A, \mathbb{C})$, while classes of even Fredholm modules over $A$ form the K-homology group $K^0(A) : = KK^0(A, \mathbb{C})$. 

For any C$^*$-algebra $A$ and $i=0,1$ there are pairings between $K_i(A)$ and $K^i(A)$, see \cite{Connes}. For our purpose we only need the pairing of the class in $K_0(A)$ of a projection $P$ in the algebra $A$ and a Fredholm module $(H_{ev} \oplus H_{odd}, \rho_{ev} \oplus \rho_{odd}, \F)$. It is given by 
$$
\langle [P]_0 , (H_{ev} \oplus H_{odd}, \rho_{ev} \oplus \rho_{odd}, \F) \rangle = \textrm{Index}(\rho_{ev}(P) G \rho_{odd}(P)), 
$$
where the index above is of the restricted operator
$$\rho_{ev}(P) G \rho_{odd}(P): \textrm{Ran}( \rho_{odd}(P)) \to \textrm{Ran}(\rho_{ev}(P))\,.$$ 

\subsection{Constructions}
First we describe a simple Fredholm module the class of which  in $K$-homology 
$K^0(HS(s))$ is the homomorphism $\eta$ of equation \eqref{etadef}.  In this one-dimensional Fredholm module the even Hilbert space is just zero while the odd is $\C$. The even representation of $HS(s)$ and the operator $G$ must then be zero while the odd representation is defined by:
\begin{equation*}
V\mapsto 1\textrm{ and }M_f\mapsto f(0).
\end{equation*}
Notice that this representation is zero on the ideal $I_s$ and in particular it is zero on the generators of $K_0(HS(s))$ coming from that ideal. It follows that the corresponding index is zero. On the other hand, the range of the odd representation is one-dimensional which implies that the index of $G$ is 1. All that means is that the pairing of this Fredhoilm module with $K_0(HS(s))$ is precisely $\eta$.

Next, given a homomorphism $\Phi\in \Hom(C_{\{1\}}(\Zsstar, \Z),\Z)$ such that
\begin{equation}\label{PhiDecomp}
\Phi=\sum_{y\in T}\phi_{(y)}\,e_{(y)}= \sum_{y\in T}\phi_{(y)}(\delta_y-\delta_{\gamma(y)})
\end{equation}
we want to describe an even Fredholm module $(H_{ev} \oplus H_{odd}, \rho_{ev} \oplus \rho_{odd}, \F)$ such that its pairing with $K_0(HS(s))\cong C_{\{1\}}(\Zsstar, \Z)\oplus\Z$ gives the homomorphism $\Phi\oplus 0$. 
Informally, we build the representation $\rho_{ev} \oplus \rho_{odd}$ from evaluations at $y$ and $\gamma(y)$ for those $y$ for which $\phi_{(y)}\ne 0$ and group the terms in equation \eqref{PhiDecomp} with positive coefficients to form $H_{odd}$ and the terms with negative coefficients to form $H_{ev}$.

More precisely, the Hilbert space for this module is a separable Hilbert space with distinguished basis labeled in the following way. $H_{odd}$ has basis $E^{{odd},+}_{(y,j,l)}$ for all $y\in T$ such that $\phi_{(y)}>0$, $j=1,\ldots,\phi_{(y)}$ and $l\in\Z_{\geq 0}$. Other basis elements are $E^{{odd},-}_{(z,k,l)}$ for all $z\in T$ such that $\phi_{(z)}<0$ and $k=1,\ldots,|\phi_{(y)}|$ and $l\in\Z_{\geq 0}$. Similarly, the distinguished basis in $H_{ev}$ is labeled $E^{{ev},+}_{(y,j,l)}$ for all $y$ such that $\phi_{(y)}>0$ and $j=1,\ldots,\phi_{(y)}$, $l\in\Z_{\geq 0}$, and also $E^{{ev},-}_{(z,k,l)}$ for all $z$ such that $\phi_{(z)}<0$, $j=1,\ldots,|\phi_{(y)}|$ and $l\in\Z_{\geq 0}$.

Since $HS(s)$ is a crossed product by an endomorphism:
\begin{equation*}
HS(s) = C(\Z_s)\rtimes_\alpha\N,
\end{equation*}
to define a representation of $HS(s)$ we need to represent its generators and verify that they satisfy the relations of Stacey's crossed product given by equation \eqref{StaceyRel}. With this in mind, the representation $\rho_{ev}$ of $HS(s)$ in $B(H_{ev})$ is defined on basis elements as follows:
\begin{equation*}
\rho_{ev}(M_f)E^{{ev},+}_{(y,j,l)} = f(s^l\gamma(y)) E^{{ev},+}_{(y,j,l)},\ \ \ \  \rho_{ev}(M_f)E^{{ev},-}_{(z,k,l)}= f(s^lz) E^{{ev},-}_{(z,k,l)},
\end{equation*}
and
\begin{equation*}
\rho_{ev}(V)E^{{ev},+}_{(y,j,l)} = E^{{ev},+}_{(y,j,l+1)},\ \ \ \  \rho_{ev}(V)E^{{ev},-}_{(z,k,l)}= E^{{ev},-}_{(z,k,l+1)}.
\end{equation*}

The definition of $\rho_{odd}$ is similar:
\begin{equation*}
\rho_{odd}(M_f)E^{{odd},+}_{(y,j,l)} = f(s^ly) E^{{odd},+}_{(y,j,l)}, \ \ \ \ \rho_{odd}(M_f)E^{{odd},-}_{(z,k,l)}= f(s^l\gamma(z)) E^{{odd},-}_{(z,k,l)}
\end{equation*}
and
\begin{equation*}
\rho_{odd}(V)E^{{odd},+}_{(y,j,l)} = E^{{odd},+}_{(y,j,l+1)},\ \ \ \  \rho_{odd}(V)E^{{odd},-}_{(z,k,l)}= E^{{odd},-}_{(z,k,l+1)}.
\end{equation*}

It ie easy to see that the relations in equation  \eqref{StaceyRel} are satisfied, and so the above formulas indeed define a continuous representation of the C$^*$-algebra $HS(s)$.

We also consider an operator $G: H_{odd}\to H_{ev}$ given on basis elements by:
\begin{equation*}
GE^{{odd},+}_{(y,j,l)} = E^{{ev},+}_{(y,j,l)} \textrm{ and } GE^{{odd},-}_{(z,k,l)}=E^{{ev},-}_{(z,k,l)}.
\end{equation*}
The definition implies that $G$ is an isometric isomorphism of Hilbert spaces:
\begin{equation*}
G^*G=I_{H_{odd}} \textrm{ and } GG^*=I_{H_{ev}}.
\end{equation*}

\begin{theo}\label{FredModTheo}
With the above notation, $(H_{ev} \oplus H_{odd}, \rho_{ev} \oplus \rho_{odd}, \F)$ where $\F$ is given by  equation \eqref{FGmatrix}, is an even Fredholm module over $HS(s)$. The class of this module in $K$-homology 
$K^0(HS(s))\cong \Hom(C_{\{1\}}(\Zsstar, \Z),\Z)\oplus\Z$ is the homomorphism $\Phi\oplus 0$.
\end{theo}
\begin{proof} 

From the definitions it follows that $\F^2=I$ and $\F$ is self-adjoint.  Thus, to check that  $(H_{ev} \oplus H_{odd}, \rho_{ev} \oplus \rho_{odd}, \F)$ is a Fredholm module we need to verify that $[\F, \rho(a)] \in \mathcal{K}(H)$ for all $a\in HS(s)$. 

From the structure of $HS(s)$ any element is of the form $a=T(\phi)+c$ with $\phi\in C(\R/\Z)$ and $c\in I_s$. This task can be simplified in the following way.  Since compact operators form a closed ideal in $B(H)$, by continuity it is enough to verify that $[\F, \rho(a)] \in \mathcal{K}(H)$ for $a$ in a dense subalgebra of $HS(s)$. A convenient choice are finite sums of the form
\begin{equation*}
a = \sum_{m\geq 0}\phi_mV^m + \sum_{m<0}\phi_m(V^*)^{-m}\,+\,\sum_{n\ge0}V^nM_{F_n} + \sum_{n<0}M_{F_n}(V^*)^{-n},
\end{equation*}
see equations \eqref{IsSeries} and \eqref{Tseries}.

Next, notice that $\F$ and $\rho(V)$ commute:
\begin{equation*}
[\F,\rho(V)]=0.
\end{equation*}
This is because, on basis elements, $\rho(V)$ shifts the $l$ index while formulas for $G$ and $G^*$ are independent of $l$.

Further simplification comes from the observation that the commutator with $\F$ is a derivation and consequently, the Leibniz rule implies that it is enough to check that $[\F, \rho(a)] \in \mathcal{K}(H)$ for algebraic generators of the subalgebra.

For $p\ge 0$ let $\chi_p\in c_0(\Z_{\ge0})$ be a sequence given by $\chi_p=\{\delta_{mp}\}_{m\ge 0}$. 
Notice that if 
$$f(s^ml)=1_{(x)}(l) \chi_p(m) \textrm{ and } f(0)=0,$$
with $l\in \Zsstar$, then 
\begin{equation*}
M_f=m_{1_{(x)}} \mu_{\chi_p}.
\end{equation*}
Notice also that 
\begin{equation}\label{chizero}
I-VV^*=\mu_{\chi_0}.
\end{equation}

It follows from the structure of the ideal $I_s$ and formulas \eqref{Mphi} and \eqref{Muchi} that the operators of the form:
\begin{equation*}
m_{1_{(x)}}V^{n} \mu_{\chi_p} \textrm{ and } m_{1_{(x)}}\mu_{\chi_p}(V^*)^n,
\end{equation*}
$x\in T$, $p\in\Z_{\ge 0}$, form generators of  the ideal $I_s$. In view of our previous remarks on Leibniz rule for the commutator with $\F$ as well as the fact that $\F$ and $\rho(V)$ commute we reduce our problem to verifying that 
\begin{equation*}
[\F, \rho(m_{1_{(x)}} \mu_{\chi_p})] \in \mathcal{K}(H).
\end{equation*}

A direct calculation for any $a\in HS(s)$ yields
\begin{equation*}
[\F,\rho(a)] = \begin{pmatrix} 0 & G\rho_{odd}(a)-\rho_{ev}(a)G \\ G^*\rho_{ev}(a)-\rho_{odd}(M_f)G^* & 0\end{pmatrix}.
\end{equation*}
Thus, we only have to check that  $G\rho_{odd}(m_{1_{(x)}} \mu_{\chi_p})-\rho_{ev}(m_{1_{(x)}} \mu_{\chi_p})G$ is a compact operator from $H_{odd}$ to $H_{ev}$.  

Notice on the ``odd'' basis elements we have the following formulas:
\begin{equation*}
\begin{aligned}
&(G\rho_{odd}(M_f)-\rho_{ev}(M_f)G)E_{(y,j,l)}^{odd,+} = (f(s^ly)-f(s^l\gamma(y)))E_{(y,j,l)}^{ev,+}\\
&(G\rho_{odd}(M_f)-\rho_{ev}(M_f)G)E_{(z,k,l)}^{odd,-} = (f(s^l\gamma(z))-f(s^lz))E_{(z,k,l)}^{ev,-}.
\end{aligned}
\end{equation*}

We study the above equation for $M_f=m_{1_{(x)}} \mu_{\chi_p}$ and the ``odd/plus'' basis elements. By Proposition \ref{edeltaprop}, we have that
\begin{equation*}
\begin{aligned}
\chi_p(l)(1_{(x)}(y)-1_{(x)}(\gamma(y)))E_{(y,j,l)}^{ev,+} &= (\delta_y(1_{(x)})-\delta_{\gamma(y)}(1_{(x)}))E_{(y,j,p)}^{ev,+} = e_{(y)}(1_{(x)})E_{(y,j,p)}^{ev,+} \\
&=\left\{\begin{aligned}
&E_{(x,j,p)}^{ev,+} &&\textrm{ if }y=x \textrm{ and }l=p\\
&0 &&\textrm{ else.}
\end{aligned}\right.
\end{aligned}
\end{equation*}
The same happens with the ``odd/minus'' basis elements except for a difference of a minus sign.  Thus $G\rho_{odd}(m_{1_{(x)}} \mu_{\chi_p})-\rho_{ev}(m_{1_{(x)}} \mu_{\chi_p})G$ is a finite rank operator and in particular, it is compact. This establishes that $(H_{ev} \oplus H_{odd}, \rho_{ev} \oplus \rho_{odd}, \F)$
 is an even Fredholm module over $HS(s)$.

Finally, we need to verify that the index homomorphism for this Fredholm module coincides with $\Phi\oplus 0$. It is enough to verify that on the identity and on the other generators $m_{1_{(x)}} \mu_{\chi_0}$'s of $K_0(HS(s))$ where $x\in T$, see formulas \eqref{Kgenerators} and \eqref{chizero}.  

Since the operator $G: H_{odd}\to H_{ev}$ is invertible it has zero index and so the Fredholm module $(H_{ev} \oplus H_{odd}, \rho_{ev} \oplus \rho_{odd}, \F)$ pairs trivially with the class of identity in $K_0(HS(s))$. Non-trivial pairing happens on the other generators of $K_0(HS(s))$ as computed below.

Let the operator
\begin{equation*}
B_x:\textrm{Ran}(\rho_{odd}(m_{1_{(x)}} \mu_{\chi_0}))\to\textrm{Ran}(\rho_{ev}(m_{1_{(x)}} \mu_{\chi_0}))
\end{equation*}
be given by 
$$B_x = \rho_{ev}(m_{1_{(x)}} \mu_{\chi_0})G\rho_{odd}(m_{1_{(x)}} \mu_{\chi_0}).$$ 
We want to compute its index. We first look at its domain.

Using the definition of $\rho$ on the ``odd'' basis elements yields
\begin{equation*}
\rho_{odd}(m_{1_{(x)}} \mu_{\chi_0})E_{(y,j,l)}^{odd,+} = \delta_y(1_{(x)})E_{(y,j,0)}^{odd,+}\quad\textrm{and}\quad \rho_{odd}(m_{1_{(x)}} \mu_{\chi_0})E_{(z,k,l)}^{odd,-}=\delta_{\gamma(z)}(1_{(x)})E_{(z,k,0)}^{odd,-}.
\end{equation*}
Using Proposition \ref{edeltaprop} we see that
\begin{equation*}
\delta_y(1_{(x)})=(e_{(y)}+e_{(\gamma(y))}+\cdots+e_{(0)})(1_{(x)}) \textrm{ and } \delta_{\gamma(z)}(1_{(x)})= (e_{(\gamma(z))}+e_{(\gamma^2(z))}+\cdots+e_{(0)})(1_{(x)}).
\end{equation*}
Consequently, we have
\begin{equation*}
\textrm{Ran}(\rho_{odd}(m_{1_{(x)}} \mu_{\chi_0})) = \textrm{span}(X_1\cup X_2)
\end{equation*}
where
\begin{equation*}
X_1 = \{E_{(y,j,0)}^{odd,+}:x=\gamma^q(y),\textrm{ for some }q\ge0\}\textrm{ and }X_2 =\{E_{(z,k,0)}^{odd,-}: x=\gamma^q(z),\textrm{ for some }q\ge1\}.
\end{equation*}

Calculating on the ``odd'' basis elements yields
\begin{equation*}
B_xE_{(y,j,0)}^{odd,+} = \delta_{\gamma(y)}(1_{(x)})\delta_y(1_{(x)})E_{(y,j,0)}^{ev,+}\quad\textrm{and}\quad B_xE_{(z,k,0)}^{odd,-}=\delta_z(1_{(x)})\delta_{\gamma(z)}(1_{(x)})E_{(z,k,0)}^{ev,-}.
\end{equation*}

We can now determine the kernel of $B_x$ on the range of $\rho_{odd}(m_{1_{(x)}} \mu_{\chi_0})$. If $x=y$, it follows that 
$$e_{(\gamma(y))}(1_{(x)})=e_{(\gamma^2(y))}(1_{(x)})=\cdots=e_{(0)}(1_{(x)})=0.$$ 
On the other hand, if $x=\gamma^q(y)$ for $q\ge1$ then 
$$\delta_y(1_{(x)})\delta_{\gamma(y)}(1_{(x)}) \neq 0.$$  
Consequently we have that the kernel of $B_x$ restricted to the range of $\rho_{odd}(m_{1_{(x)}} \mu_{\chi_0})$ is equal to the following set:
\begin{equation*}
\textrm{span}\{E_{(y,j,0)}^{odd,+}:j=1,\ldots \varphi_x\}.
\end{equation*}
Thus we have that
\begin{equation*}
\textrm{dim}(\textrm{Ker}(B_x)) = \left\{
\begin{aligned}
&\varphi_x &&\textrm{ if }\varphi>0\\
&0 &&\textrm{ if }\varphi<0.
\end{aligned}\right.
\end{equation*}

A completely analogous calculation shows that
\begin{equation*}
\textrm{dim}(\textrm{Ker}(B_x^*)) = \left\{
\begin{aligned}
&0 &&\textrm{ if }\varphi>0\\
&-\varphi_x &&\textrm{ if }\varphi<0.
\end{aligned}\right.
\end{equation*}
Putting these together shows that $\textrm{ind}(B_x)=\varphi_x=\Phi(1_{(x)})$, as claimed.
\end{proof}

\section{Spectral Triples}

The unbounded analogues of Fredholm modules are called spectral triples \cite{Rn}.
They are important tools in noncommutative geometry. In the definition of a spectral triple, as compared to a Fredholm module, we usually suppress mentioning the Hilbert space $H$ as it is implicit in the definition of a representation $\rho$. The new element is a dense $*$-subalgebra $\mathcal{A}$ of a C$^*$-algebra $A$.  See reference \cite{Ren} for a discussion of smooth subalgebras associated to spectral triples. 

\subsection{Definitions}

A spectral triple for a unital $C^*$-algebra $A$ is a triple $(\mathcal A, \rho, \mathcal D)$ where $\rho: A \rightarrow B(H)$ is a representation of $A$ on a Hilbert space $H$,  $\mathcal A$ is a dense $*$-subalgebra of $A$, and $\mathcal D$ is an unbounded self-adjoint operator in $H$ satisfying: 

(1) for every $a \in \mathcal A$, $\rho(a)$ preserves the domain of $\mathcal D$,

(2) for every $a \in \mathcal A$, the commutator $[\mathcal D, \rho(a)]$ is bounded,

(3) the resolvent $(1+ \mathcal D^2)^{-1/2}$ is a compact operator.
\newline Moreover,  $(\mathcal A, H, \mathcal D)$ is said to be an even spectral triple if there a $\Z_2$-grading of $H$, such that $H=H_{ev} \oplus H_{odd}$, $\rho_{ev} \oplus \rho_{odd}$ and 
\begin{equation*}
\mathcal D = \begin{bmatrix} 0 & D \\D^* & 0 \end{bmatrix}.
\end{equation*}

An even spectral triple for a unital $C^*$-algebra $A$ determines a homomorphism from $K_0(A)$ to $\Z$. For our purpose we only need the formula for this homomorphism for the $K_0(A)$ class of a projection $P$ in the algebra $A$. As for Fredholm modules it is given by an index 
\begin{equation*}
 P\mapsto \textrm{Index}(\rho_{ev}(P) D \rho_{odd}(P)).
\end{equation*}

\subsection{Constructions}

Let $\Phi\in \Hom(C_{\{1\}}(\Zsstar, \Z),\Z)$ be a homomorphism given by equation \eqref{PhiDecomp}. Also, let  $H_{ev} \oplus H_{odd}$ and $\rho_{ev} \oplus \rho_{odd}$ be the Hilbert space and the representation from Theorem \ref{FredModTheo}. We use those objects to define a spectral triple over $HS(s)$ corresponding to $\Phi\oplus 0$.

Consider an operator $D: \textrm{dom }D\subset H_{odd}\to H_{ev}$ given on basis elements by:
\begin{equation*}
DE^{{odd},+}_{(y,j,l)} = \Lambda(y,l) E^{{ev},+}_{(y,j,l)} \textrm{ and } DE^{{odd},-}_{(z,k,l)}=\Lambda(z,l) E^{{ev},-}_{(z,k,l)},
\end{equation*}
with $\Lambda(y,l)$ to be determined later and $\textrm{dom }D$ is the maximal domain of $D$. We need 
$$\Lambda(y,l)\to \infty$$ 
for the resolvent $(1+ \mathcal D^2)^{-1/2}$ to be a compact operator.

Consider first the question of boundedness of the commutators in two special cases: $[\mathcal{D}, \rho(V)]$ and $[\mathcal{D}, \rho(M_F)]$ with $F\in C_0(\Z_{\geq 0}\times\Zsstar)$ and is Lipschitz in the second coordinate. Those commutators are off diagonal with respect to the grading on $H$ and the matrix elements can be explicitly computed. Indeed, we have:
\begin{equation*}
( D\rho_{odd}(V)-\rho_{ev}(V)D)E^{{odd},+}_{(y,j,l)}=(\Lambda(y,l+1)-\Lambda(y,l))E^{{ev},+}_{(y,j,l+1)}
\end{equation*}
and similarly for $E^{{odd},-}_{(z,k,l)}$. It follows that
\begin{equation}\label{STnormest1}
\|[\mathcal{D}, \rho(V)]\|=\underset{y\in T,l\in\Z_{\ge0}}{\textrm{sup}}\,|\Lambda(y,l+1)-\Lambda(y,l)|.
\end{equation}
The other commutator calculation is analogous:
\begin{equation*}
( D\rho_{odd}(M_F)-\rho_{ev}(M_F)D)E^{{odd},+}_{(y,j,l)}=\Lambda(y,l)(F(l,y)-F(l,\gamma(y)))E^{{ev},+}_{(y,j,l)}
\end{equation*}
and similarly for $E^{{odd},-}_{(z,k,l)}$. It follows that
\begin{equation*}
\|[\mathcal{D}, \rho(M_F)]\|=\underset{y\in T,l\in\Z_{\ge0}}{\textrm{sup}}\,|\Lambda(y,l)|\,|F(l,y)-F(l,\gamma(y))|\leq \underset{y\in T,l\in\Z_{\ge0}}{\textrm{sup}}\,|\Lambda(y,l)|\,\|F(l,\cdot)\|_L|y-\gamma(y)|_s,
\end{equation*}
where we used Lipschitz continuity of $F$.
From the definition, if $s^{n(y)-1}\leq y<s^{n(y)}$ then $\gamma(y)= y \textrm{ mod } s^{n(y)-1}$ and so 
$$|y-\gamma(y)|_s\leq s^{-(n(y)-1)},$$
which gives
\begin{equation}\label{STnormest2}
\|[\mathcal{D}, \rho(M_F)]\|\leq \underset{y\in T,l\in\Z_{\ge0}}{\textrm{sup}}\,|\Lambda(y,l)|\,\|F(l,\cdot)\|_L\, s^{-(n(y)-1)}.
\end{equation}

A simple choice of $\Lambda$ so that the norms in equations \eqref{STnormest1} and \eqref{STnormest2} are finite and $\Lambda(y,l)\to \infty$ is
\begin{equation}\label{lambdaequ}
\Lambda(y,l)=c_1 l+c_2 s^{n(y)}
\end{equation}
for any positive constants $c_1$, $c_2$.

\begin{theo}
With the above notation $(HS_L^\infty(s), \rho_{ev} \oplus \rho_{odd}, \mathcal{D})$ is an even spectral triple.  The class of this spectral triple in $K$-homology 
$K^0(HS(s))\cong \Hom(C_{\{1\}}(\Zsstar, \Z),\Z)\oplus\Z$ is the homomorphism $\Phi\oplus 0$.
\end{theo}
\begin{proof} 
Notice that $I_{s,L}^\infty$ is a Fr\'{e}chet algebra with respect to the following usual RD norms
\begin{equation*}
\|a\|_N:=\sum_{n\in\Z,\,m\in\Z_{\geq0}}(1+|m|+|n|)^N\,\|F_n(m,\cdot)\|_L.
\end{equation*}
where $a$ is given by the series from equation \eqref{IsSeries}. Similarly, $C^\infty(\R/\Z)$ is a Fr\'{e}chet space with norms
\begin{equation*}
\|\phi\|_N:=\sum_{m\in\Z}(1+|m|)^N\,|\phi_m|.
\end{equation*}
where $\phi_m$ are  Fourier coefficients of $\phi\in C^\infty(\R/\Z)$. We say that an element of $HS_L^\infty(s)$ is {\it polynomial} if it is given by a finite sum in formulas \eqref{IsSeries} and \eqref{Tseries}. Notice that the polynomial elements of $HS_L^\infty(s)$ form a dense subalgebra.

By the definition of the maximal domain of $D$ we have that
\begin{equation*}
\begin{aligned}
&\textrm{dom }D =\{\Omega\in H_{odd}: \|D\Omega\|^2<\infty\}=\\
&\left\{\Omega= \sum_{y,j,l}\omega_{(y,j,l)}^+E_{(y,j,l)}^{odd,+} + \sum_{z,k,l}\omega_{(z,k,l)}^-E_{(z,k,l)}^{odd,-} : \sum_{y,j,l}\Lambda^2(y,l)|\omega_{(y,j,l)}^+|^2 + \sum_{z,k,l}\Lambda^2(z,l)|\omega_{(z,k,l)}^-|^2<\infty\right\}
\end{aligned}
\end{equation*}
and similarly for the adjoint. $\textrm{dom }D$ is a Hilbert space with the inner product given by $\|D\Omega\|^2$ in the above expression. Notice that finite linear combinations of the basis elements form a dense subspace in $\textrm{dom }D$.

We need to verify that $\rho(a)$ preserves the domain of $\mathcal D$  and $[\mathcal{D}, \rho(a)]$ is bounded for all $a \in  HS_L^\infty(s)$. Notice that the polynomial elements of $HS_L^\infty(s)$ preserve the dense subspace of finite linear combinations of the basis elements. Since we have
\begin{equation*}
\|\mathcal{D}\rho(a)\Omega\|\leq \|[\mathcal{D},\rho(a)]\Omega\|+\|\rho(a)\mathcal{D}\Omega\|\leq \|[\mathcal{D}, \rho(a)]\|\|\Omega\|+\|a\|\|\mathcal{D}\Omega\|,
\end{equation*}
we see that estimates on $[\mathcal{D}, \rho(a)]$ in terms of Fr\'{e}chet norms of $a$ would imply that $\rho(a)$ preserves the domain of $\mathcal D$ by an approximation argument.

For $\Lambda$ given by formula \eqref{lambdaequ} we have by equation \eqref{STnormest1}:
\begin{equation*}
\|[\mathcal{D}, \rho(V)]\|\leq c_1
\end{equation*}
which implies that 
\begin{equation*}
\|[\mathcal{D}, \rho(V^m)]\| \leq c_1|m|
\end{equation*}
for every $m\in\Z$.
Consequently, for $\phi\in C^\infty(\R/\Z)$ and the Toeplitz operator $T(\phi)$ given by equation \eqref{Tseries} we get
\begin{equation*}
\|[\mathcal{D}, \rho(T(\phi))]\| \leq c_1\sum_{m\in\Z}|m|\,|\phi_m|\leq c_1\|\phi\|_1.
\end{equation*}

Inserting $\Lambda$ from formula \eqref{lambdaequ} into inequality \eqref{STnormest2} we can estimate as follows:
\begin{equation*}
\|[\mathcal{D}, \rho(M_F)]\|\leq \underset{l\in\Z_{\ge 0}}{\textrm{sup}}\,s(c_1l+c_2)\|F(l,\cdot)\|_L \leq \sum_{l\in\Z_{\geq 0}}s(c_1l+c_2)\|F(l,\cdot)\|_L .
\end{equation*}
Consequently, for $a\in I_s$ given by equation \eqref{IsSeries} we obtain
\begin{equation*}
\|[\mathcal{D}, \rho(a)]\|\leq \sum_{n\in\Z,\,l\in\Z_{\geq0}}(s(c_1l+c_2)+c_1|n|)\,\|F_n(l,\cdot)\|_L\leq \textrm{const}\|a\|_1.
\end{equation*}
By the previous remarks, those estimates prove that $(HS_L^\infty(s), \rho_{ev} \oplus \rho_{odd}, \mathcal{D})$ is an even spectral triple.

Finally, the index calculation for the spectral triple is completely analogous to the index calculation for the Fredholm module in the proof of Theorem \ref{FredModTheo} as the extra non-zero $\Lambda(y,l)$ factor does not change the kernel and cokernel of the corresponding operators.   

\end{proof}

\end{document}